\documentclass[11pt]{article}
\usepackage[T1]{fontenc}
\usepackage{lmodern}
\usepackage[colorlinks=true,linkcolor=RoyalBlue,urlcolor=RoyalBlue,citecolor=red]{hyperref}
\usepackage{amsmath,amsthm,amssymb,amsfonts}
\usepackage{enumerate}
\usepackage{epsfig}
\usepackage[dvipsnames,svgnames,table]{xcolor}
\usepackage{pgf,tikz}
\usetikzlibrary{calc}
\usepackage{fullpage}

\usepackage{booktabs}
\usepackage[nameinlink]{cleveref}
\usepackage{enumitem}
\usepackage{blkarray}
\usepackage{tikz-cd}

\usepackage{thmtools, thm-restate}

\newtheorem{theorem}{Theorem}[section]
\newtheorem{proposition}[theorem]{Proposition}
\newtheorem{lemma}[theorem]{Lemma}
\newtheorem{corollary}[theorem]{Corollary}

\newtheorem{claim}[theorem]{Claim}

\theoremstyle{definition}

\theoremstyle{remark}
\newtheorem{remark}[theorem]{Remark}

\newenvironment{claimproof}[1][\proofname]
	{
		\proof[#1]%
			
	}
	{
		\endproof
	}

\newcommand{\CC}{\mathbb{C}}
\newcommand{\RR}{\mathbb{R}}
\newcommand{\QQ}{\mathbb{Q}}

\newcommand{\FF}{\mathbb{F}}

\renewcommand{\SS}{\mathbb{S}}

\DeclareMathOperator{\td}{td}
\DeclareMathOperator{\cliq}{Cliq}

\tikzstyle{vertex}=[fill=black,circle,inner sep=0pt, minimum size=4pt]
\tikzstyle{edge}=[line width=1.5pt,black]
\tikzstyle{bedge}=[line width=1.5pt,RoyalBlue]
\tikzstyle{dedge}=[line width=1.5pt,black,dashed]

\title{The number of realisations of a random graph}
\author{Sean Dewar\thanks{Department of Computer Science, KU Leuven. E-mail: \texttt{sean.dewar@kuleuven.be}}  \and Anthony Nixon\thanks{School of Mathematical Sciences, Lancaster University. E-mail: \texttt{a.nixon@lancaster.ac.uk}}  \and Ben Smith\thanks{School of Mathematical Sciences, Lancaster University. E-mail: \texttt{b.smith9@lancaster.ac.uk}}}

\begin{document}
\date{}
\maketitle

\begin{abstract}
Determining the number of realisations, up to isometries, of a graph for a specific choice of edge lengths is a fundamental problem in discrete geometry. 
In this article we prove that, asymptotically almost surely, the $d$-dimensional complex realisation number for each $n$-vertex graph in an Erd\H{o}s--R\'enyi random graph process is either infinite or equal to $2^{n-t}$ where $t$ is the size of the $(d+1)$-core; moreover this number coincides exactly with the real realisation number for such graphs.
We also determine a similar formula for the number of complex solutions to the generic rank-$d$ positive semi-definite matrix completion problem with randomly selected non-diagonal unknown entries.
\end{abstract}



\section{Introduction}\label{sec:intro}

Given a finite, simple graph $G=(V,E)$,
a map $\mathbf{p} :V \rightarrow \mathbb{R}^d$ with $\mathbf{p}(v) = (\mathbf{p}_1(v), \ldots, \mathbf{p}_d(v))$ for each $v \in V$ is called a \emph{realisation of $G$ in $\mathbb{R}^d$}.
Two realisations $\mathbf{p},\mathbf{q}$ of $G$ in $\mathbb{R}^d$ are said to be \emph{equivalent} if they satisfy the following equalities:
\begin{equation} \label{eq:edge+length}
    \sum_{i=1}^d \Big( \mathbf{p}_i(v) - \mathbf{p}_i(w) \Big)^2 = \sum_{i=1}^d \Big( \mathbf{q}_i(v) - \mathbf{q}_i(w) \Big)^2 \qquad \text{ for all } vw \in E,
\end{equation}
i.e.~any edge $vw$ has the same Euclidean length in $\mathbf{p}$ and $\mathbf{q}$.
With this, we fix $r_d(G,\mathbf{p}) \in \mathbb{N} \cup \{\infty\}$ to be the number of equivalent realisations of $G$ in $\mathbb{R}^d$ modulo isometries of $\mathbb{R}^d$, i.e.~combinations of rotations, reflections and translations.
Asimow and Roth \cite{AsimowRothI} proved that if $\mathbf{p}$ is a \emph{generic} realisation of $G$ in $\mathbb{R}^d$ (as defined in \Cref{subsec:realisation}) then $r_d(G,\mathbf{p})$ is finite if and only $r_d(G,\mathbf{q})$ is finite for every other generic realisation $\mathbf{q}$ of $G$ in $\mathbb{R}^d$.
Because of this, we say that $G$ is \emph{$d$-rigid} if $r_d(G,\mathbf{p})$ is finite for some generic realisation of $G$ in $\mathbb{R}^d$.

In contrast to rigidity being a generic property,
the exact value of $r_d(G,\mathbf{p})$ (when finite) is often not consistent over all generic realisations.
Indeed,
it is not hard to construct small examples of graphs that have different numbers of realisations depending on the edge lengths chosen; see for example \Cref{fig:notgeneric}.
To get around this, we can associate a single `generic realisation number' to a graph in two different ways:
\begin{enumerate}
    \item We define the \emph{real $d$-realisation number} -- here denoted $r_d(G)$ -- to be the largest value of $r_d(G,\mathbf{p})$ over all generic realisations $p$ in $\mathbb{R}^d$.
    \item We extend $r_d(G,\mathbf{p})$ to additionally count complex realisations.
    We let $c_d(G,\mathbf{p})$ be the number of complex realisations $\mathbf{q}\colon V \rightarrow \CC^d$ that satisfy \eqref{eq:edge+length} modulo (complex) rotations, reflections and translations. 
    The value $c_d(G,\mathbf{p})$ \emph{is} a generic property, hence we define the \emph{(complex) $d$-realisation number} -- here denoted $c_d(G)$ -- to be $c_d(G,\mathbf{p})$ for any generic realisation $\mathbf{p}$ in $\mathbb{C}^d$. 
\end{enumerate}
It is immediate that $c_d(G) \geq r_d(G)$ for all graphs, and it was proven by Gortler and Thurston \cite{Gortler2014} that $r_d(G)=1$ (i.e.~$G$ is \emph{globally $d$-rigid}) if and only if $c_d(G)=1$.

\begin{figure}[tp]
	\begin{center}
        \begin{tikzpicture}[scale=0.6]
        \begin{scope}[xshift=0]
			\node[vertex] (1) at (0,1) {};
			\node[vertex] (2) at (0,-1) {};
			
			\node[vertex] (3) at (-1.5,-0.5) {};
			\node[vertex] (4) at (1.5,0.5) {};
			\node[vertex] (5) at (-0.5, 1.5) {};
			
			\draw[edge] (1)edge(2);
			\draw[edge] (1)edge(3);
			\draw[edge] (1)edge(4);
			\draw[edge] (2)edge(3);
			\draw[edge] (2)edge(4);
			\draw[edge] (5)edge(3);
			\draw[edge] (5)edge(4);
        \end{scope}
		\begin{scope}[xshift=100]
			\node[vertex] (1) at (0,1) {};
			\node[vertex] (2) at (0,-1) {};
			
			\node[vertex] (3) at (1.5,-0.5) {};
			\node[vertex] (4) at (1.5,0.5) {};
			\node[vertex] (5) at (3.679,0) {};
			
			\draw[edge] (1)edge(2);
			\draw[edge] (1)edge(3);
			\draw[edge] (1)edge(4);
			\draw[edge] (2)edge(3);
			\draw[edge] (2)edge(4);
			\draw[edge] (5)edge(3);
			\draw[edge] (5)edge(4);
		\end{scope}
        \begin{scope}[xshift=275]
            \draw[line width=1pt] (0,1.5) -- (0,-1.5);
		\end{scope}
        \begin{scope}[xshift=350]
			\node[vertex] (1) at (0,1) {};
			\node[vertex] (2) at (0,-1) {};
			
			\node[vertex] (3) at (1.5,-0.5) {};
			\node[vertex] (4) at (1.5,0.5) {};
			\node[vertex] (5) at (2,0) {};
			
			\draw[edge] (1)edge(2);
			\draw[edge] (1)edge(3);
			\draw[edge] (1)edge(4);
			\draw[edge] (2)edge(3);
			\draw[edge] (2)edge(4);
			\draw[edge] (5)edge(3);
			\draw[edge] (5)edge(4);
        \end{scope}
		\end{tikzpicture}
	\end{center}
    \caption{Two different choices of edge lengths for the same graph. The left choice gives $r_2(G,\mathbf{p}) = 4$: the two realisations pictured plus two more via reflecting the degree 2 vertex through the line passing through its neighbours. The right choice only gives $r_2(G,\mathbf{p})=2$: again, the additional realisation can be found via reflecting the degree 2 vertex.}\label{fig:notgeneric}
\end{figure}
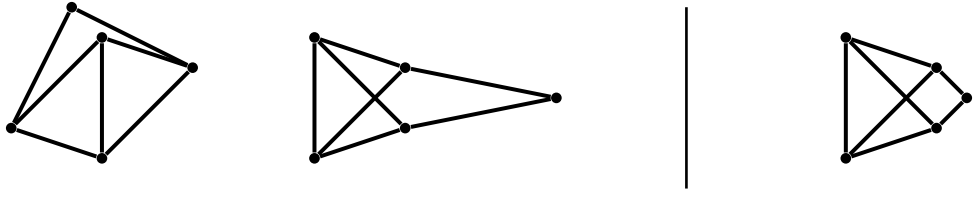

Outside of the almost-trivial case of $d=1$,
both real and complex realisation numbers are often difficult to compute.
For example, using Gr\"{o}bner basis to compute $c_d(G)$ very quickly becomes computationally untenable.
A combinatorial algorithm for computing $c_2(G)$ for minimally 2-rigid graphs is known \cite{cggkls},
which was recently extended to all 2-rigid graphs \cite{dgstw25}.
Unfortunately, no such combinatorial algorithms are known for computing $c_d(G)$ when $d \geq 3$ or $r_d(G)$ when $d \geq 2$.
Many techniques have been developed in the last two decades to provide exact values and bounds for both $c_d(G)$ and $r_d(G)$:
see for example \cite{BARTZOS2021189,BS04,JO19,Stef10}.

Since computing the (real) $d$-realisation number of a specific graph can be difficult,
a simpler question one may ask is: how do (real) $d$-realisation numbers behave for \emph{most} graphs
with some bounded number of edges?
These types of questions are usually formalised using one of the \emph{Erd\H{o}s--R\'enyi random graph models}:
\begin{enumerate}
    \item \emph{$G(n,M)$ model:} a graph with vertex set $[n]$ and exactly $M$ edges that is chosen uniformly at random from the set of $M$-edge graphs with vertex set $[n]$.
    \item \emph{$G(n,p)$ model:} a graph with vertex set $[n]$ where each distinct pair of vertices is adjacent with probability $p$, independent of all other distinct pairs.
\end{enumerate}

A stronger question we can ask is:
how do the (real) $d$-realisation numbers change as we add edges at random?
In this case we turn to the \emph{Erd\H{o}s--R\'enyi random graph process}, which we define as follows.
Given a fixed value $n$, choose a total ordering $\omega$ of the edges of $K_n$\footnote{Here each total ordering is a bijective map $\omega: E(K_n) \rightarrow \{1, \ldots, \binom{n}{2}\}$, where $E(K_n)$ is the edge set of $K_n$.} uniformly at random from the set of all such total orderings.
For each integer $0\leq M \leq \binom{n}{2}$, fix $G(n,M,\omega)$ to be the graph with vertex set $[n]$ and edge set $\{ vw : \omega(vw) \leq M\}$.
More on Erd\H{o}s--R\'enyi random graph processes can be found in \cite[Section 2]{Frieze_K_2015}.

Any graph that is $d$-rigid with at least $d+1$ vertices must have minimum degree at least $d$ (see, for example, \cite[Lemma 2.6.1.b(ii)]{GraverServatius}).
Similarly, it follows from a result of Hendrickson \cite[Theorem 5.9]{Hen92} that any graph that is globally $d$-rigid with at least $d+2$ vertices must have minimum degree at least $d +1$.
Lew, Nevo, Peled and Raz~\cite{lewetal} proved that these are sharp thresholds, and that with high probability, $d$-rigidity (respectively, global $d$-rigidity) occurs as soon as the minimum degree is at least $d$ (respectively, $d+1$):

\begin{theorem}[Lew, Nevo, Peled and Raz~\cite{lewetal}]\label{thm:lnpr}
    Let $\omega$ be chosen uniformly at random from the set of total orderings of the edges of $K_n$.
    Then the following property holds a.a.s.~(with respect to $n$):
    for any integer $0 \leq M \leq \binom{n}{2}$ and graph $G=(G,M,\omega)$:
    \begin{enumerate}
        \item $G$ is $d$-rigid if and only if the minimum degree of $G$ is at least $d$, and
        \item $G$ is globally $d$-rigid if and only if the minimum degree of $G$ is at least $d+1$.
    \end{enumerate}
\end{theorem}

In this paper we prove both that the real and complex realisation numbers coincide for such random graphs and that the realisation number is directly tied to the size of its \emph{$k$-core}: the unique maximal subgraph of $G$ in which every vertex has degree at least $k$.

\begin{theorem}\label{thm:random}
    Let $\omega$ be chosen uniformly at random from the set of total orderings of the edges of $K_n$.
    Then the following property holds a.a.s.~(with respect to $n$):
    for any integer $0 \leq M \leq \binom{n}{2}$,
    if the $(d+1)$-core of $G = G(n,M,\omega)$ contains exactly $t$ vertices,
    then
    \begin{equation*}
        c_d(G) = r_d(G) = 
        \begin{cases}
            2^{n-t} &\text{if $G$ has minimum degree at least $d$,}\\
            \infty &\text{otherwise}.
        \end{cases}
    \end{equation*}
\end{theorem}

\Cref{thm:random} immediately implies the following for the two Erd\H{o}s--R\'enyi random graph models.

\begin{corollary}\label{cor:random}
    Let $(a_n)_{n \in \mathbb{N}}$ be a real sequence where $a_n \rightarrow \infty$ as $n \rightarrow \infty$.
    Suppose that either:
    \begin{enumerate}
        \item $G=G(n,p)$ with $p = (\log n + (d-1) \log \log n + a_n)/n$; or
        \item $G=G(n,M)$ with $M =  \lceil \frac{1}{2}(\log n + (d-1) \log \log n + a_n)n \rceil$.
    \end{enumerate}
    Then, given the $(d+1)$-core of $G$ contains exactly $t$ vertices,
    the following equality holds a.a.s.:
    \begin{equation*}
        c_d(G) = r_d(G) = 2^{n-t}.
    \end{equation*}
\end{corollary}

\Cref{thm:random} can be viewed as the realisation number analogue of the sharp rigidity threshold described by Lew, Nevo, Peled and Raz (\Cref{thm:lnpr}).
Indeed, our proof utilises a key `expansion' lemma from their paper (see \Cref{prop:neighbourlyprop}).
To finish the proof, we use Vill\'{a}nyi's recent ground-breaking result \cite{vill}, which in particular resolved positively the Lov\'{a}sz--Yemini conjecture, as well as classical results of {\L}uczak regarding the size and connectivity of $k$-cores in random graphs \cite{luczak}.

An important consequence of \Cref{thm:random} is that for most graphs we have $r_d(G)$ and $c_d(G)$ are equal; moreover, both values are equal to some power of 2 for which the exponent can be computed in quadratic time. 

\subsection{Random PSD matrix completion}

The combinatorics employed in the proof of \Cref{thm:random} can be adapted to solve a rather different problem involving positive semidefinite (PSD) matrix completions.

Consider the following matrix completion problem.
Given a fixed graph $G=([n],E)$ and a field $\FF \in \{\RR,\CC\}$,
we define a \emph{partial $n \times n$ symmetric matrix (with known diagonal)} 
to be any vector 
\begin{equation*}
    A = \Big((A_{i,j})_{\{i,j\} \in E}, (A_{k,k})_{k \in [n]} \Big) \in \FF^{|E|} \times \FF^n.
\end{equation*}
We say that $G$ is the \emph{underlying graph} of $A$.
We often consider the vector $A$ to be a symmetric $n \times n$ matrix where the diagonal and the entries indexed by $E$ are \emph{known entries} and the others are \emph{unknown entries}.
A \emph{rank-$d$ partial $n \times n$ PSD matrix} is any partial matrix over $\RR$ for which there exists an $n \times n$ PSD matrix $B$ with rank at most $d$ which agrees with the known entries of $A$.
We say that any such PSD matrix $B$ is a \emph{completion} of $A$.
We are interested in finding and counting completions.
As in the case of realisation numbers, we will restrict our investigation to \emph{generic} partial PSD matrices (as defined in \Cref{subsec:partialmatrix}).

In \cite{SingerCucuringu:2010},
Singer and Cucuringu proved that a generic rank-$d$ partial PSD matrix has finitely many completions if and only if its underlying graph is $(d-1)$-rigid.
However, similar to the realisation counting problem, the number of completions is not a generic property.
Because of this, we instead allow \emph{complex completions}:
any symmetric matrix $B \in \mathbb{C}^{n \times n}$ which agrees with the known entries of $A$ and can be factored as $B = P^\top P$ for some matrix $P \in \CC^{d \times n}$.
With this, the number of complex completions of a generic rank-$d$ partial PSD matrix is dependent solely on its underlying graph (See, for example, \Cref{prop:matrix+completion} below).

For integers $0 < d \leq n$, we create a \emph{random rank-$d$ partial $n \times n$ PSD matrix process} as follows:
\begin{itemize}
    \item Sample a random $d \times n$ matrix $P$\footnote{Formally, $P$ is a continuous random variable whose sample space has positive measure in the set of $d \times n$ matrices.}, and set $B = P^\top P$ to be the resulting random $n \times n$ PSD matrix of rank $d$.
    
    \item Sample a total ordering $\omega$ of the edges of $K_n$ uniformly at random.
    \item For each integer $0 \leq M \leq \binom{n}{2}$, define the projection map $\pi_M$ that maps symmetric $n\times n$ matrices $(b_{i,j})_{i,j \in [n]}$ to $((b_{ii})_{i \in [n]},(b_{ij})_{ij \in E})$, where $E$ is the edge set of the graph $G(n,M,\omega)$.
    \item Output the $n \times n$ partial PSD matrices $A_0,\ldots, A_{\binom{n}{2}}$ with $A_M=\pi_M(B)$.
\end{itemize}
Sampling random PSD matrices is well studied in random matrix theory: for example, if the rows of $P$ are drawn from a normal distribution, then $B$ is sampled from a \emph{Wishart distribution} \cite{wishart}. 
Moreover, they occur naturally as covariance matrices in a wide variety of applications including neuroscience, communications and machine learning.
It is common in real-world problems for only a subset of the entries to be observable, and so matrix completion techniques for such matrices have become an important and well-studied topic \cite{BGN17,BY14}.

\begin{theorem}\label{thm:partialmatrix}
    Let $A_0,\ldots, A_{\binom{n}{2}}$ be a random rank-$d$ partial $n \times n$ PSD matrix process.
    Then a.a.s.~the following holds for each matrix $A_M$:
    \begin{enumerate}
        \item If each row and column of $A_M$ contains at least $d$ known entries then there exist $2^{n-t}$ complex completions of $A_M$,
        where $t$ is the size of the $(d+1)$-core of $A_M$ (as defined in \Cref{subsec:partialmatrix}).
        \item If each row and column of $A_M$ contains less than $d$ known entries then there exist infinitely-many complex completions of $A_M$.
    \end{enumerate}
\end{theorem}

\subsection{Layout of paper}

We first review notation, terminology and provide preliminary results for realisation numbers, random graphs and rigidity theory in \Cref{sec:prelims}.
In \Cref{sec:random} we establish precisely the $d$-dimensional realisation number of an Erd\H{o}s--R\'enyi random graph (\Cref{thm:random}).
We conclude the paper in \Cref{sec:matrix} by adapting results from \Cref{sec:random} to spherical realisation numbers and the PSD matrix completion problem.

\section{Preliminaries}
\label{sec:prelims}

In this section we present the required background on random graphs, realisation numbers and globally linked pairs.

\subsection{Random graph terminology}\label{subsec:randomgraphs}

We start with random graph models. See \cite{Frieze_K_2015} for more background on the topic.

An important topic for Erd\H{o}s--R\'enyi random graph processes is the  minimum degree threshold. First, let $B_n$ be the set of bijective maps $E(K_n) \rightarrow \{1, \dots, \binom{n}{2}\}$, and
let $\delta(G)$ denote the minimum degree in $G$.
For each $\omega \in B_n$,
we define the value
\begin{equation*}
    M_d(\omega) := \textrm{min} \left\{ 0 \leq M \leq \binom{n}{2}: \delta(G(n,M,\omega)) \geq d \right\}.
\end{equation*}
This variable tells us if we were to pick edges in the order given by $\omega$, then our graph would have minimum degree $d$ for the first time exactly when it has $M_d(\omega)$ edges.

We often make use of the following terminology to describe `guaranteed' properties of a random graph.
A property $\mathcal{P}$ which is well-defined for each set $S_n$ in a sequence of sets $(S_n)_{n \in \mathbb{N}}$ is said to hold \emph{asymptotically almost surely} (a.a.s.) if 
\begin{equation*}
    \lim_{n \rightarrow \infty} \mathbb{P} [ \mathcal{P} \text{ holds for } s \in S_n] = 1.
\end{equation*}

\begin{remark}
    As stated in \Cref{sec:intro},
    if a graph with at least $d+1$ (respectively, $d+2$) vertices is $d$-rigid (respectively, globally $d$-rigid) then $\delta(G) \geq d$ (respectively, $\delta(G) \geq d+1$).
    This immediately implies that for any $\omega \in B_n$,
    if $G(n,M,\omega)$ is $d$-rigid (respectively, globally $d$-rigid) then $M \geq M_d(\omega)$ (respectively, $M \geq M_{d+1}(\omega)$).
    Hence \Cref{thm:lnpr} is equivalent to the equalities
    \begin{equation*}
        \lim_{n \rightarrow \infty} \mathbb{P} \Big[ G\big(n,M_d(\omega),\omega \big) \text{ is $d$-rigid} \Big] = \lim_{n \rightarrow \infty} \mathbb{P} \Big[ G\big(n,M_{d+1}(\omega),\omega \big) \text{ is globally $d$-rigid}  \Big]= 1
    \end{equation*}
    where each probability involves sampling $\omega$ uniformly from $B_n$.
\end{remark}

\subsection{Realisation numbers}\label{subsec:realisation}

Although computing realisation numbers can be difficult, we will only require some basic background results regarding them to prove our main results.
We direct any reader wishing to learn more about the topic to \cite{dewar23,JO19,dewarsmithnixon2026}, and additionally suggest \cite[Section 10]{Gortler2014} for its discussion on congruent realisations.

We begin by defining `genericity' in our context.
Given fields $K \subset L$, we write $\td[L \colon K]$ for the transcendence degree of $L$ over $K$.
We say a realisation $p$ of $G$ in either $\mathbb{R}^d$ or $\mathbb{C}^d$ is \emph{generic} if $\td[\QQ(p) \colon \QQ] = d|V|$, where $\QQ(p)$ is the field extension of $\QQ$ generated by $\{p_i(v): i \in [d], v \in V\}$.
This is equivalent to the coordinates of $p$ being algebraically independent over $\mathbb{Q}$.
By our choice of definition for genericity, every generic realisation in $\mathbb{R}^d$ is also a generic realisation in $\mathbb{C}^d$.

For a field $\FF \in \{\RR,\CC\}$,
we define $O(d,\FF)$ to be the group of all matrices $A \in \FF^{d \times d}$ where $A^\top A = I_d$, with $I_d$ being the $d\times d$ identity matrix.
Likewise, we define $E(d,\FF)$ to be the group of affine transformations of $\FF^d$ of the form $x \mapsto Ax +y$ with $A \in O(d,\FF)$ and $y \in \FF^d$.
When we equip the vector space $\FF^d$ with the quadratic form
\begin{equation*}
    \|(x_1,\ldots,x_d)\|^2 := \sum_{i=1}^d x_i^2,
\end{equation*}
the group $E(d,\FF)$ is exactly the group of isometries of $(\FF^d,\|\cdot\|^2)$, i.e.~maps $f:\FF^d \rightarrow \FF^d$ for which $\|f(x) - f(y)\|^2 = \|x-y\|^2$ for all points $x,y \in \FF^d$.
This follows from \cite[Corollary 8]{Gortler2014}.

We define two realisations $p,q$ of a graph $G$ in $\FF^d$ to be \emph{congruent} if there exists an isometry $f$ such that $\mathbf{q} = f \circ \mathbf{p}$.
This forms an equivalence relation which we denote by $\sim$.
We define $\mathbf{p},\mathbf{q}$ to be \emph{equivalent} if $\|\mathbf{p}(v) - \mathbf{p}(w)\|^2=\|\mathbf{q}(v) - \mathbf{q}(w)\|^2$ holds for all edges $vw \in E$.
We define for any realisation $\mathbf{p}$ of $G$ in $\CC^d$ the value
\begin{equation*}
    c_d(G,\mathbf{p}) := \# \Big( \left\{ q \text{ a realisation of $G$ in $\CC^d$ that is equivalent to $p$ }   \right\} / \sim \Big) \in \mathbb{N} \cup \{\infty\}.
\end{equation*}
Similarly, if $p$ is a realisation of $G$ in $\RR^d$, define the value
\begin{equation*}
    r_d(G,\mathbf{p}) := \# \Big( \left\{ q \text{ a realisation of $G$ in $\RR^d$ that is equivalent to $p$ }   \right\} / \sim \Big) \in \mathbb{N} \cup \{\infty\}.
\end{equation*}
Unlike $r_d(G,\mathbf{p})$, the value $c_d(G,\mathbf{p})$ is the same for any choice of generic $p$ in $\CC^d$.
We define the \emph{$d$-realisation number} and \emph{real $d$-realisation number} of $G$ to be the respective values
\begin{align*}
    c_d(G) &:= c_d(G,\mathbf{p}) \qquad \qquad \text{ for any generic realisation $p$ of $G$ in $\CC^d$}, \\
    r_d(G) &:= \max \Big\{ r_d(G,\mathbf{p}) : \text{ generic realisation $p$ of $G$ in $\RR^d$} \Big\}.
\end{align*}
Since every generic realisation in $\RR^d$ is also generic in $\CC^d$, it is clear that $r_d(G) \leq c_d(G)$.

A graph $G$ is said to be \emph{$d$-rigid} (respectively, \emph{globally $d$-rigid}) if $r_d(G,\mathbf{p})$ is finite (respectively, $r_d(G,\mathbf{p})=1$) for some generic $\mathbf{p}$ in $\RR^d$.
Since rigidity is a generic property (see \cite{AsimowRothI}), $G$ is $d$-rigid if and only if $r_d(G) \leq c_d(G) < \infty$.
Gortler, Healy and Thurston \cite{gortler2010characterizing} proved that global rigidity is also a generic property,
and so $G$ is globally $d$-rigid if and only if $r_d(G)=1$.
This was later extended to complex realisations by Gortler and Thurston \cite{Gortler2014}, who showed that $c_d(G)=1$ if and only if $r_d(G)=1$.

The simplest graph operation that preserves $d$-rigidity is the \emph{$d$-dimensional 0-extension}.
This operation on $G$ adds a new vertex $x$ and $d$ new edges $xw_1, \dots, xw_d$ to distinct vertices $w_i$ of $G$.
A schematic of a 3-dimensional 0-extension is given in \Cref{fig:0-extension}.
In addition to preserving rigidity, it precisely doubles both the real and complex $d$-realisation numbers.

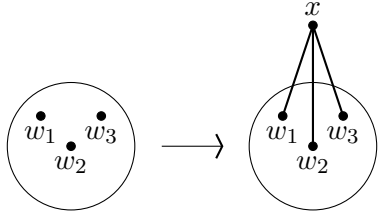
\begin{figure}
    \centering
    \begin{tikzpicture}[scale=0.8]

    \begin{scope}
        \draw[line width=1.5pt] (0,0) circle (30pt);
        
        \node[vertex] (w1) at (0.5,0.5) {};
        \node[vertex] (w2) at (0,0) {};
        \node[vertex] (w3) at (0.5,-0.5) {};
        
        \node (w1label) at (0,0.5) {$w_1$};
        \node (w2label) at (-0.5,0) {$w_2$};
        \node (w3label) at (0,-0.5) {$w_3$};
    \end{scope}

    \begin{scope}[xshift=75]
        \draw[thick,->] (0,0) -- (2,0);
    \end{scope}

    \begin{scope}[xshift=200]
        \draw[line width=1.5pt] (0,0) circle (30pt);
        
        \node[vertex] (w1) at (0.5,0.5) {};
        \node[vertex] (w2) at (0,0) {};
        \node[vertex] (w3) at (0.5,-0.5) {};
        \node[vertex] (x) at (2.5,0) {};
        
        \node (w1label) at (0,0.5) {$w_1$};
        \node (w2label) at (-0.5,0) {$w_2$};
        \node (w3label) at (0,-0.5) {$w_3$};
        \node (xlabel) at (3,0) {$x$};

        \draw[edge] (w1)edge(x);
        \draw[edge] (w2)edge(x);
        \draw[edge] (w3)edge(x);
    \end{scope}
\end{tikzpicture}
    \caption{A schematic of a 3-dimensional 0-extension.}
    \label{fig:0-extension}
\end{figure}












\begin{lemma}[{\cite[Lemma 7.1]{dewar23}}]\label{lem:0ext}
    Let $G = (V,E)$ be a $d$-rigid graph with at least $d+1$ vertices.
    If $G'$ is obtained from $G$ via a 0-extension, then $c_d(G') = 2c_d(G)$.
\end{lemma}

\begin{proposition}[{\cite[Proposition 4.5]{dewarsmithnixon2026}}]\label{prop:0extreal}
    Let $G = (V,E)$ be a $d$-rigid graph with at least $d+1$ vertices.
    If $G'$ is obtained from $G$ via a 0-extension, then $r_d(G') = 2r_d(G)$.
\end{proposition}

\subsection{Weakly globally linked pairs}

A vertex pair $\{v,w\}$ of a graph $G$ is said to be \emph{weakly globally $d$-linked} if there exists a generic realisation $\mathbf{p}$ of $G$ in $\mathbb{R}^d$ where $\|\mathbf{p}(v) - \mathbf{p}(w)\|=\|\mathbf{q}(v) - \mathbf{q}(w)\|$ holds for each equivalent realisation $\mathbf{q}$ of $G$ in $\mathbb{R}^d$.
We connect weakly globally $d$-linked vertex pairs to the generic property of global $d$-rigidity using the following lemma.

\begin{lemma}[{\cite[Lemma 3.2(a)]{Jordán2024}}]\label{lem:jv}
    Let $G$ be a graph and let
    \begin{equation*}
        J_d(G) := \Big\{ vw : \{v,w\} \text{ is weakly globally $d$-linked and non-adjacent in } G  \Big\}.
    \end{equation*}
    If $G+J_d(G)$ is globally $d$-rigid,
    then $G$ is globally $d$-rigid.
\end{lemma}

Let $G=(V,E)$ be a graph with vertex subset $X \subset V$.
Let $V_1,\ldots, V_m \subset V$ denote the vertex sets of the connected components of $G-X$. Then define $\cliq (G,X)$ to be the graph formed from $G$ by deleting the vertex sets $V_i$ for $1 \leq i \leq m$, and adding an edge $vw$ for all non-adjacent vertex pairs $v,w$ where both $v$ and $w$ both have a neighbour in the same component $V_i$.

\begin{theorem}[{\cite[Theorem 4.5]{Jordán2024}}]\label{thm:jv}
    Let $G=(V,E)$ be a $(d+1)$-connected graph with non-adjacent vertices $x,y$,
    and let $V_0 \subset V$ be a vertex set containing both $x,y$.
    If $G[V_0]$ is $d$-rigid and $\cliq (G,V_0)$ is globally $d$-rigid,
    then $\{x,y\}$ is weakly globally $d$-linked in $G$.
\end{theorem}

\section{Proof of \texorpdfstring{\Cref{thm:random}}{main theorem}}
\label{sec:random}

Here we denote the $k$-core of a graph $G$ by $G(k)$, and we denote its vertex set by $V(k)$.

Our proof of \Cref{thm:random} proceeds as follows:
\begin{enumerate}
    \item First, in \Cref{lem:weirdlem} we show that every graph $G$ with the $d$-neighbourly property (defined in \Cref{subsec:neighbourly}) and a large $d(d+1)$-core has a vertex ordering $v_1,\ldots,v_n$ such that
    $G[\{v_1,\ldots,v_s\}]= G(d(d+1))$,
    and the remaining vertices $v_{s+1},\ldots,v_n$ are added one-by-one by connecting the vertex to at least $d$ previous vertices.
    Furthermore, our ordering will ensure that there exists $t \geq s$ such that $G[\{v_1,\ldots,v_t\}]$ is the $(d+1)$-core of $G$,
    and the vertices $v_{t+1}, \ldots,v_n$ are added by $d$-dimensional 0-extensions.
    \item Next, in \Cref{lem:0extplusedgesisgr} we show that any graph formed from a globally $d$-rigid graph by performing a sequence of $d$-dimensional 0-extensions and then adding a set of edges so that the resulting graph is $(d+1)$-connected is globally $d$-rigid.
    \item Finally, we prove that if $\omega \in B_n$ is chosen uniformly at random then a.a.s.~$G(n,M,\omega)$ satisfies the hypotheses for \Cref{lem:weirdlem} whenever $M \geq M_d(\omega)$.
    We then combine these results with a recent theorem of Vill\'{a}nyi (\Cref{thm:soma}) to prove \Cref{thm:random}.
\end{enumerate}

\subsection{The \texorpdfstring{$d$}{d}-neighbourly property}\label{subsec:neighbourly}

We say that a graph $G=(V,E)$ has the \emph{$d$-neighbourly property} if every subset $B \subset V$ with $1 \leq |B| \leq |V|/2$ contains a vertex with at least $d$ neighbours in $V\setminus B$.

When combined with a large $d(d+1)$-core, the $d$-neighbourly property guarantees that the graph can be constructed in a specific way that will be useful.

\begin{lemma}\label{lem:weirdlem}
Let $d\geq 1$ and set $k=d(d+1)$.
    Suppose that $G$ is an $n$-vertex graph with the $d$-neighbourly property and a $k$-core with at least $5n/9$ vertices.
    Then there exists an ordering $v_1,\ldots,v_n$ of the vertices of $G$ and integers $s \leq t\leq n$ such that the following hold:
    \begin{enumerate}
        \item \label{lem:weirdlem1} $v_1,\ldots,v_s$ are the vertices of $G(k)$.
        \item \label{lem:weirdlem2} $v_{s+1}, \ldots, v_t$ are exactly the vertices contained in $G(d+1)$ but not $G(k)$.
        \item \label{lem:weirdlem3} For $s+1 \leq i \leq n$, the vertex $v_i$ is adjacent to at least $d$ vertices in the set $\{v_1,\ldots,v_{i-1}\}$.
    \end{enumerate}
\end{lemma}

\begin{proof}
    We first construct an ordering $w_1,\ldots,w_n$ of $[n]$ as follows.
    First we choose $w_1,\ldots, w_s$ to be the vertices of $G(k)$ with an arbitrary ordering.
    For $i \geq s+1$, we now inductively choose a vertex $w_i$ from $[n] \setminus \{w_1,\ldots,w_{i-1}\}$ with at least $d$ neighbours in $\{w_1,\ldots,w_{i-1}\}$ by using the $d$-neighbourly property with $B= [n] \setminus \{w_1,\ldots,w_{i-1}\}$.
    Let $\vec{G}$ be the acyclic directed graph formed from $G$ by setting $w_i \rightarrow w_j$ if $i > j$.
    By our choice of ordering, we have that every vertex outside of $G(k)$ has out-degree at least $d$.
    We now form the ordering $v_1,\ldots,v_n$ as follows.
    \begin{enumerate}[label=(\arabic*)]
        \item For each $i \leq s$,
        set $v_i = w_i$.
        \item Starting at $i=n$ and proceeding all the way down to $i=s+1$,
        we choose $v_i$ to be the vertex not contained in the set $\{v_1,\ldots,v_s\}$ with in-degree 0 and the smallest out-degree in the directed graph $\vec{G}_i := \vec{G} - \{v_{i+1},\ldots, v_n\}$; in the case of a tie, just choose one vertex arbitrarily.
        (The existence of vertices of in-degree 0 is guaranteed by choice of construction.)
    \end{enumerate}
    Consider the sequence of directed graphs $\vec{G}_n, \vec{G}_{n-1}, \ldots, \vec{G}_{s+1}$.
    The directed graph $\vec{G}_i$ is obtained from $\vec{G}_{i+1}$ by deleting the vertex $v_{i+1}$.
    Since $v_{i+1}$ has in-degree 0 in $\vec{G}_{i+1}$,
    it follows that the out-degree of every vertex in $\vec{G}_i$ is the same as its out-degree in $\vec{G}_{i+1}$.
    It thus follows that the out-degree of any vertex in $\vec{G}_i$ is the same as its out-degree in $\vec{G}$ for any $s+1 \leq i \leq n$.
    Since each of the vertices $v_{s+1}, \ldots, v_n$ has out-degree at least $d$ in $\vec{G}$,
    our sequence $v_1,\ldots, v_n$ satisfies conditions \ref{lem:weirdlem1} and \ref{lem:weirdlem3}.

    Fix $t$ to be the size of the $(d+1)$-core of $G$.
    If $t=n$ there is nothing left to prove, so we may suppose that $t < n$.
    It now suffices to show that each vertex $v_{t+1}, \ldots, v_n$ has out-degree $d$ in $\vec{G}$;
    indeed, if true this would imply that $G - \{v_{t+1}, \ldots, v_n\} = G(d+1)$, and thus our choice of $v_1,\ldots,v_n$ satisfies \ref{lem:weirdlem2}.
    Suppose for the sake of contradiction that this is not true.
    Fix $t+1 \leq i \leq n$ to be the largest index for which $v_i$ has out-degree larger than $d$ but $v_j$ has out-degree $d$ for each $j >i$.
    Then the graph $H:= G - \{v_{i+1},\ldots,v_n\}$ is formed by deleting a sequence of vertices of degree $d$.
    Since $|V(H)| = i >t$,
    $H$ is not the $(d+1)$-core of $G$ but does contain the $(d+1)$-core of $G$ (since we have only deleted vertices of degree at most $d$),
    thus $H$ must have a vertex $v$ of degree at most $d$.
    Since each vertex $v_1, \ldots,v_s$ is contained in $G(k)$ (and thus must have degree at least $k$),
    there exists some $s+1 \leq j \leq i$ so that $v=v_j$.
    Thus $v$ has out-degree at least $d$ in $\vec{G}$.
    Since $v$ has degree at most $d$ and out-degree at least $d$,
    $v$ has out-degree $d$ and in-degree 0.
    But this implies that the out-degree of $v$ is smaller than that of $v_i$,
    contradicting how we chose which of $w_{s+1}, \ldots,w_n$ was $v_i$ earlier.
    This contradiction completes the proof.
\end{proof}

\subsection{A specific type of globally rigid graph}

Our next step for proving \Cref{thm:random} is the following key lemma regarding global $d$-rigidity.

\begin{lemma}\label{lem:0extplusedgesisgr}
    Let $G'$ be formed from a globally $d$-rigid graph by a sequence of 0-extensions, and let $G$ contain $G'$ as a spanning subgraph.
    If $G$ is $(d+1)$-connected, then $G$ is globally $d$-rigid.
\end{lemma}

To prove \Cref{lem:0extplusedgesisgr}, we first require the following folklore result.

\begin{lemma}\label{lem:d+1dim0ext}
    Let $H$ be globally $d$-rigid and let $G$ be formed from $H$ by adding a vertex of degree at least $d+1$.
    Then $G$ is globally $d$-rigid.
\end{lemma}

\begin{proof}
    This is a direct consequence of \cite[Lemma 4.1]{tan15}.
\end{proof}

Our approach to proving \Cref{lem:0extplusedgesisgr} is to use \Cref{thm:jv} to prove all pairs of vertices are weakly globally $d$-linked, then use \Cref{lem:jv} to achieve global rigidity.

\begin{proof}[Proof of \Cref{lem:0extplusedgesisgr}]
    Let $\mathcal{G}_k(K)$ be the family of all $(d+1)$-connected graphs $G$ containing a spanning subgraph $H$ formed in the following way:
    \begin{enumerate}
        \item Start with a globally rigid graph $K$ with at least $d+1$ vertices.
        \item Perform a sequence of exactly $k$ $d$-dimensional 0-extensions.
    \end{enumerate}
    We now refer to $H$ as a \emph{$k$-spine} of $G$.
    Our aim now is to prove that for each $k$ and globally $d$-rigid graph $K$, every graph in $\mathcal{G}_k(K)$ is globally $d$-rigid.
    We do so via induction on $k$.

    Choose any $G \in \mathcal{G}_1(K)$ with 1-spine $H$ formed by adding a single vertex $v_1$ adjacent to $d$ vertices in $K$.
    Since $H$ is not $(d+1)$-connected,
    $G$ contains at least one additional edge between $v_1$ and another vertex $u$ in $K$.
    By \Cref{lem:d+1dim0ext}, $H + v_1 u$ is globally $d$-rigid,
    and hence $G$ is globally $d$-rigid.

    Now suppose that for any choice of globally $d$-rigid graph $K$, every graph in $\mathcal{G}_k(K)$ is globally $d$-rigid.
    Choose any $G \in \mathcal{G}_{k+1}(K)$ with $(k+1)$-spine $H$.
    Specifically,
    let $H$ be the graph formed from $K$ by a sequence of $d$-dimensional 0-extensions which add vertices $v_0, v_1, \ldots, v_k$ in that order.

    \begin{claim}
        For any vertex $u$ in $K$ that is not adjacent to $v_0$ in $G$,
        the pair $\{v_0,u\}$ is weakly globally $d$-linked in $G$.
    \end{claim}

    \begin{claimproof}
        Choose any vertex $u \in V(K)$ that is not adjacent to $v_0$ in $G$.
        Set $V_0 = V(K) \cup \{v_0\}$.
        By our choice of $V_0$,
        the induced subgraph $G[V_0]$ contains a $1$-spine as a spanning subgraph and hence is $d$-rigid.
        Since $G$ is $(d+1)$-connected, there exist $d+1$ internally
        disjoint paths between $v_0$ and $u$.
        Let $w_1,\ldots,w_d$ be the neighbours of $v_0$ in $H$ that are contained in $K$.
        Since $G$ is $(d+1)$-connected, 
        there exists a path $Q$ from $v_0$ to $u$ in $G$ that does not pass through $w_1,\ldots,w_d$.
        Let $z$ be the first vertex in $Q$ that is contained in $V(K)$ and let $P$ be the path formed by truncating $Q$ at $z$.
        It follows that the vertex $v_0$ is adjacent to at least $d+1$ vertices in the graph $\cliq(G,V_0)$.
        
        Let $K'$ be $\cliq (G,V_0)$ restricted to the vertices of $K$.
        As $K'$ contains $K$ as a spanning subgraph, $K'$ is globally $d$-rigid.
        Thus $\cliq (G,V_0)$ can be formed from $K'$ by  adding a vertex of degree $d+s$ for some $s \geq 1$, and hence the graph $\cliq(G,V_0)$ is globally $d$-rigid by \Cref{lem:d+1dim0ext}.
        The claim now follows from \Cref{thm:jv}.
    \end{claimproof}

     Now define the graphs
     \begin{align*}
         H' &:= H + \{ v_0 u : v_0 u \notin E(H), ~ u \in V(K) \} + \{ uv : uv \notin E(H), ~ u, v \in V(K)\}, \\
         G' &:= G + \{ v_0 u : v_0 u \notin E(G), ~ u \in V(K) \} + \{ uv : uv \notin E(G), ~ u, v \in V(K)\}.
     \end{align*}
     Both $G', H'$ contain the complete graph $K_{V_0}$, and have only had edges added between weakly globally linked pairs of vertices.
     With this,
     it is clear that the graph $G'$ is contained in $\mathcal{G}_k(K_{V_0})$ and has $k$-spine $H'$ obtained from $K_{V_0}$.
     By our inductive argument, $G'$ is globally $d$-rigid.
     Since $G'$ is a spanning subgraph of $G+J_d(G)$,
     the latter graph $G+J_d(G)$ is globally $d$-rigid.
     Hence, $G$ is globally $d$-rigid by \Cref{lem:jv}.
\end{proof}

\subsection{Tying it all together}

We first need to show that graphs of minimum degree at least $d$
obtained from an Erd\H{o}s--R\'enyi random graph process satisfy the hypotheses of \Cref{lem:weirdlem} a.a.s.
Let us begin with the easier part: the $d$-neighbourly property.
It is clear that any graph with the $d$-neighbourly property must have minimum degree $d$ (i.e.~take $B$ to be any single element set).
Lew, Nevo, Peled and Raz proved that the converse also holds a.a.s.~for Erd\H{o}s--R\'enyi random graphs.

\begin{proposition}[{\cite[Proposition 2.3]{lewetal}}]\label{prop:neighbourlyprop}
    Let $\omega \in B_n$ be chosen uniformly at random.
    Then a.a.s.~the following holds: for all $M \geq M_d(\omega)$, the graph $G(n,M,\omega)$ has the $d$-neighbourly property. 
\end{proposition}

We obtain the large $k$-connected $k$-cores we need by adapting classical results of {\L}uczak \cite{luczak}.

\begin{theorem}\label{thm:luczaknew}
    Let $\omega \in B_n$ be chosen uniformly at random.
    Then a.a.s.~the following holds for fixed $k \geq 3$: for each $M \geq M_2(\omega)$, the $k$-core of $G(n,M,\omega)$ is $k$-connected and contains at least $5n/9$ vertices.
\end{theorem}

This result mostly follows from adapting {\L}uczak's results from $G(n,p)$ to Erd\H{o}s--R\'enyi random graph processes.
As their proofs are very similar, we defer the proof to \Cref{appendix}.

Before we tie everything together,
we require the following sufficient condition for global rigidity recently proven by Villányi.

\begin{theorem}[Villányi~\cite{vill}]\label{thm:soma}
    Every $d(d+1)$-connected graph is globally $d$-rigid.
\end{theorem}

With this, we are now ready to prove \Cref{thm:random}.

\begin{proof}[Proof of \Cref{thm:random}]
    The result holds for $d=1$ by \cite[Proposition 3.8]{dewar23}, so we suppose that $d \geq 2$.
    Sample $\omega \in B_n$ uniformly at random.
    We first observe that $G(n,M,\omega)$ is not $d$-rigid if $M < M_d(\omega)$,
    and so the real and complex $d$-realisation numbers for $G(n,M,\omega)$ are both $\infty$.
    By \Cref{prop:neighbourlyprop} and \Cref{thm:luczaknew} the following holds a.a.s.: for each $M \geq M_d(\omega)$, the graph $G(n,M,\omega)$ has the $d$-neighbourly property and, for each $k \geq d+1$, a $k$-connected $k$-core with at least $5n/9$ vertices.
    \Cref{lem:weirdlem} thus implies the following holds a.a.s.: for each $M \geq M_d(\omega)$, $G(n,M,\omega)$ has the vertex ordering $v_1,\ldots, v_n$ described in \Cref{lem:weirdlem}.
    We now proceed by fixing $G=G(n,M,\omega)$ for arbitrary $M \geq M_d(\omega)$ with the assumption that this vertex ordering exists for the remainder of the proof.

    Fix $k= d(d+1)$.
    Conditions \ref{lem:weirdlem1} and \ref{lem:weirdlem2} of \Cref{lem:weirdlem} imply that $G(k) = G[\{v_1,\ldots,v_s\}]$ and $G(d+1) = G[\{v_1,\ldots,v_t\}]$.
    By \Cref{thm:soma},
    $G(k)$ is globally $d$-rigid.
    \Cref{lem:0extplusedgesisgr} and condition \ref{lem:weirdlem3} of \Cref{lem:weirdlem} now imply that $G(d+1)$ is also globally $d$-rigid.
    Conditions \ref{lem:weirdlem2} and \ref{lem:weirdlem3} of \Cref{lem:weirdlem} imply that $G$ is formed from $G(d+1)$ by a sequence of $n-t$ $d$-dimensional 0-extensions.
    The result now follows as $d$-dimensional 0-extensions exactly double the complex and real realisation number of a graph (\Cref{lem:0ext} and \Cref{prop:0extreal} respectively).
\end{proof}

\section{PSD matrix completion and spherical realisation numbers}\label{sec:matrix}

In this section we adapt the previous results to spherical realisations of graphs.
We then use a correspondence between spherical rigidity and PSD matrix completion to prove \Cref{thm:partialmatrix}.

\subsection{Spherical realisation numbers for Erd\H{o}s--R\'enyi random graphs}

We consider realisations of graphs on the sphere.
We denote the $d$-dimensional sphere by
\[
\SS_\FF^d = \{x \in \FF^{d+1} : \| x \|^2 = 1 \} \, 
\]
where $\FF\in \{\RR,\CC\}$ give the real and complex spheres respectively.
We extend the notion of genericity from $\RR^d$ to $\SS_\RR^d$, saying a realisation $\mathbf{p} \colon V \rightarrow \SS_\RR^d$ is \emph{generic} if and only if $\td[\QQ(\mathbf{p}) \colon \QQ] = d|V|$. Note that the coordinates are never algebraically independent.
Whiteley \cite{Whiteley:1983} proved that a graph $G$ is rigid when embedded generically on $\SS_\RR^d$ if and only if it is $d$-rigid. 
This equivalence has motivated research on
the $d$-dimensional \emph{spherical realisation number} $c_d^*(G)$ of a graph $G$: the number of equivalent realisations for any given generic realisation of a graph in $\SS_\CC^d$.
See for example \cite{BARTZOS2021189,dewar23,ggs20}.
As with the Euclidean case, we define the \emph{real spherical realisation number} $r_d^*(G)$ to be the maximum number of equivalent realisations as we range over all generic spherical realisations of $G$.

While there seems to be no straightforward way to determine the spherical realisation number from the Euclidean one, or vice versa, many techniques used to count realisations can often be adapted to count spherical realisations \cite{dewar23}.
In particular, we can connect the $d$-dimensional spherical realisation number of $G$ to the $(d+1)$-dimensional realisation number of the cone of $G$.
Recall that the \emph{cone} of $G$, denoted $G*o$, is obtained by adding a new vertex $o$ adjacent to all other vertices.

\begin{theorem}[\cite{dewar23}] \label{thm:spherical+complex}
    Let $G$ be a graph and $G*o$ its cone.
    Then $c_d^*(G) = c_{d+1}(G*o)$.
\end{theorem}

We can obtain a similar result for real spherical realisation numbers:

\begin{proposition} \label{prop:spherical+real}
    Let $G$ be a graph and $G*o$ its cone.
    Then $r_d^*(G) = r_{d+1}(G*o)$.
\end{proposition}

\begin{proof}
    Let $r_{d+1}(G*o) = t$ and let $\mathbf{p}^{(1)}, \dots, \mathbf{p}^{(t)}$ be equivalent but non-congruent realisations of $G*o$ in $\RR^{d+1}$ with $\mathbf{p}^{(1)}$ generic.
    By translating, we will replace for each $1 \leq i \leq t$ the realisation $\mathbf{p}^{(i)}$ with a realisation for which $\mathbf{p}^{(i)}(o)$ is placed at the origin.
    With this we have $\td[\QQ(\mathbf{p}^{(1)}) \colon \QQ] = (d+1)|V|$.
    Let $\lambda_v = \|\mathbf{p}^{(1)}(v) - \mathbf{p}^{(1)}(o)\| = \|\mathbf{p}^{(1)}(v)\|$, and note that this equality holds for all $\mathbf{p}^{(i)}$ as they are equivalent.
    Define the realisation $\tilde{\mathbf{p}}^{(i)}$ of $G$ by $\tilde{\mathbf{p}}^{(i)}(v) = \mathbf{p}^{(i)}(v) / \lambda_v$ for all $v \in V$.
    It follows that each $\tilde{\mathbf{p}}^{(i)}$ satisfies $\|\tilde{\mathbf{p}}^{(i)}(v)\| = 1$ and hence is a spherical realisation of $G$.
    The proof of \cite[Proposition 2.4]{dewar23} shows that $\tilde{\mathbf{p}}^{(1)}, \dots, \tilde{\mathbf{p}}^{(t)}$ are all equivalent.
    To show $\tilde{\mathbf{p}}^{(1)}$ is generic, suppose for contradiction that the scalars $(\tilde{\mathbf{p}}^{(i)}_{j}(v) \colon j \in [d], ~ v \in V )$ are algebraically dependent.
    As $\tilde{\mathbf{p}}^{(i)}_{j}(v)$ is an algebraic combination of $({\mathbf{p}}^{(i)}_j(v))_{j=1}^{d+1}$, this leads to an algebraic dependence between the scalars in the multiset $(\mathbf{p}^{(i)}_{j}(v) \colon j \in [d+1], ~ v \in V )$, contradicting that $\td[\QQ(\mathbf{p}^{(1)}) \colon \QQ] = (d+1)|V|$.
    It follows that $\tilde{\mathbf{p}}^{(1)}$ is generic, and hence $r_d^*(G) \geq r_{d+1}(G*o)$.

    Conversely, 
    let $r_d^*(G) = t$ and choose $\tilde{\mathbf{p}}^{(1)}, \dots, \tilde{\mathbf{p}}^{(t)}$ to be equivalent but non-congruent realisations of $G$ in $\SS_\RR^d$ with $\tilde{\mathbf{p}}^{(1)}$ generic.
    We can construct realisations of $G*o$ in $\RR^{d+1}$ by adding the new vertex $o$ at the origin.
    Pick a set $\{\lambda_v:v \in V\}$ that is algebraically independent over $\QQ(\tilde{\mathbf{p}}^{(1)})$.
    Define $\mathbf{p}^{(i)}(v) = \lambda_v \tilde{\mathbf{p}}^{(i)}(v)$ for all $1 \leq i \leq t$.
    The resulting realisations $\mathbf{p}^{(1)}, \dots, \mathbf{p}^{(t)}$ of $G*o$ are equivalent but non-congruent by \cite[Proposition 2.4]{dewar23}.
    A similar argument to the previous direction gives that $\td[\QQ(\mathbf{p}^{(1)}) \colon \QQ] = (d+1)|V|$.
    Now choose a vector $x \in \mathbb{R}^d$ whose coordinates are algebraically independent over $\QQ(\mathbf{p}^{(1)})$ and set $\mathbf{p} = (\mathbf{p}^{(1)} +x)_{v \in V \cup\{o\}}$.
    It is clear that $\mathbf{p}$ is a generic realisation of $G$ in $\mathbb{R}^d$, and hence $r_d^*(G) \leq r_d(G,\mathbf{p}) =  r_{d+1}(G*o)$.
\end{proof}

We can now adapt our methods from \Cref{sec:random} to count spherical realisations for random graphs.

\begin{corollary}\label{cor:random+spherical}
    Let $\omega$ be chosen uniformly at random from the set of total orderings of the edges of $K_n$.
    Then the following property holds a.a.s. (with respect to $n$):
    for any integer $0\leq M \leq \binom{n}{2}$,
    if the $(d+1)$-core of $G = G(n,M,\omega)$ contains exactly $t$ vertices,
    then
    \begin{equation*}
        c^*_d(G) = r^*_d(G) = 
        \begin{cases}
            2^{n-t} &\text{if $G$ has minimum degree at least $d$,}\\
            \infty &\text{otherwise}.
        \end{cases}
    \end{equation*}    
\end{corollary}

\begin{proof}
    Set $k=(d+1)(d+2)$.
    Analogous to the proof of \Cref{thm:random},
    it suffices to show the result holds for any $n$-vertex graph $G$ that has the $d$-neighbourly property and a $k$-connected $k$-core with at least $5n/9$ vertices.
    Now let $\widetilde{G}=G*o$ denote the cone of $G$, and set $\widetilde{V}$ to be the vertex set of $\widetilde{G}$.
    We observe that $\widetilde{G}$ has the $(d+1)$-neighbourly property as follows.
    If $B \subseteq \widetilde{V}$ contains $o$ then $o$ has more than $d+1$ neighbours in $\widetilde{V} \setminus B$.
    If $B \subseteq \widetilde{V}$ does not contain $o$, then pick $v \in B$ that fulfilled the $d$-neighbourly property in $G$ as this now has $o \in \widetilde{V} \setminus B$ as an additional neighbour.
    As $G$ has a $(d+1)(d+2)$-core with at least $5n/9$ vertices, $\widetilde{G}$ does also.
    We can now repeat the proof of \Cref{thm:random} to show that $c_{d+1}(\widetilde{G}) = r_{d+1}(\widetilde{G}) = 2^{|\widetilde{V} \setminus \widetilde{V}(d+2)|}$.
    The result now follows from \Cref{thm:spherical+complex}, \Cref{prop:spherical+real} and the observation that $\widetilde{V}(d+2)$ equals the cone over $V(d+1)$.
\end{proof}

\subsection{Applications to PSD matrix completion problems}\label{subsec:partialmatrix}

Throughout this section, we let $\FF \in \{\RR, \CC\}$.
We say a matrix $B \in \FF^{n \times n}$ of rank at most $d$ is \emph{PSD} if it can be factored into $B = P^\top P$ for some $P \in \FF^{d \times n}$.
When $\FF = \RR$, this is equivalent to the usual notion of a PSD matrix, though we will also use the term when $\FF = \CC$.
We note that this factorisation is not unique, and we can alternatively factor $B = Q^\top Q$ if and only if $Q = XP$ for some orthogonal matrix $X \in O(d,\FF)$.
We define a rank-$d$ PSD matrix $B = P^\top P$ to be \emph{generic} if the entries of $P$ are algebraically independent.

We similarly extend our definitions of partial matrices to allow for complex PSD matrices.
A \emph{rank-$d$ partial $n \times n$ PSD matrix} is now any partial matrix for which there exists an $n \times n$ PSD matrix $B$ with rank at most $d$ which agrees with the known entries of $A$.
We say that any such PSD matrix $B$ that can be used to construct $A$ is a \emph{completion} of $A$,
and a \emph{real completion} if each entry of $B$ is real.
A rank-$d$ partial PSD matrix is now said to be \emph{generic} if it has a generic completion.

We now show how we can view PSD matrix completion as a spherical rigidity problem; see \cite{SingerCucuringu:2010} for further details.
Write the columns of $P$ as $\mathbf{p}(1), \dots, \mathbf{p}(n) \in \FF^d$.
We normalise $A$ and $P$ by multiplying by the diagonal matrix $(A_{ii}^{-1})_{i \in [n]}$: this ensures $\|\mathbf{p}(i)\| = 1$, hence $\mathbf{p}(1), \dots, \mathbf{p}(n) \in \SS^{d-1}$ (or $\SS_\CC^{d-1}$ if $\FF = \CC$).
Moreover
\[
\|\mathbf{p}(i) - \mathbf{p}(j)\|^2 = \|\mathbf{p}(i)\|^2 + \|\mathbf{p}(j)\|^2 - 2 \mathbf{p}(i) \cdot \mathbf{p}(j) = 2 - 2\mathbf{p}(i) \cdot \mathbf{p}(j) \, .
\]
As $A_{ij} = \mathbf{p}(i) \cdot \mathbf{p}(j)$ for all $ij \in E$, the partial matrix $A$ uniquely determines edge lengths $d_{ij}$ for all $ij \in E$.
This gives a bijection between completions of $A$ and spherical realisations of $G$ with edge lengths $d_{ij}$.
Furthermore, this bijection allows us to reframe the counting problem as a spherical realisation number problem.

\begin{proposition} \label{prop:matrix+completion}
    Let $A$ be a generic rank-$d$ partial PSD matrix with underlying graph $G$.
    Then $A$ has exactly $c_{d-1}^*(G)$ complex completions.
    Moreover, if $A$ is real then it has at most $r_{d-1}^*(G)$ real completions, and there exists a real generic rank-$d$ partial PSD matrix $A'$ with exactly $r_{d-1}^*(G)$ real completions.
\end{proposition}

We will use \Cref{cor:random+spherical} and \Cref{prop:matrix+completion} to prove \Cref{thm:partialmatrix}.
Here we introduce the following analogue of graph $k$-cores for partial matrices.
The \emph{$k$-core} $A(k)$ of a partial symmetric matrix $A$ is the matrix formed by sequentially deleting rows (and their corresponding columns) with at most $k$ known entries until this is no longer possible.
Observe that if $G$ is the underlying graph of $A$, then $G(k-1)$ is the underlying graph of $A(k)$.
This follows as the $k$-core of a partial matrix also counts entries on the diagonal, while the underlying graph only sees off-diagonal entries.

\begin{proof}[Proof of \Cref{thm:partialmatrix}]
    Choose $0 \leq M \leq \binom{n}{2}$ and let $G$ be the underlying graph of $A_M$.
    Fix $\delta(A_M)$ to be the minimum number of known entries in a given row or column.
    Then $\delta(A_M)=\delta(G)+1$ and $G(d)$ is the underlying graph of $A_M(d+1)$.
    By \Cref{prop:matrix+completion}, $A_M$ has exactly $c^*_{d-1}(G)$ complex completions.
    The result now follows from applying \Cref{cor:random+spherical} with $d$ replaced by $d-1$.
\end{proof}

\begin{remark}
    All of the partial matrices we have considered require the diagonal to consist of known entries.
    This is because the reduction to a spherical rigidity problem requires us to scale each vector $p(i)$ to lie on the sphere.
    It may be possible to deduce results when a small number of diagonal entries are unknown using recent results on partial coning \cite{HJNS25}. 
\end{remark}

\subsection*{Acknowledgements}

S.\,D.\ was supported by the KU Leuven grant iBOF/23/064 and the FWO grants G0F5921N (Odysseus) and G023721N.
A.\,N.\ and B.\,S.\ were partially supported by EPSRC grant EP/X036723/1. A.\,N.\ was partially supported by UK Research and Innovation (grant number UKRI1112), under the EPSRC Mathematical Sciences Small Grant scheme.
We thank John Haslegrave for helpful discussions on the proof strategy for \Cref{thm:random} and Roan Talbut for describing how to rewrite \Cref{thm:random} to apply to Erd\H{o}s--R\'enyi random graph processes. We also acknowledge that a large language model pointed out an error in an earlier version of \Cref{lem:luczak4} in the appendix. For the purpose of open access, the author has applied a Creative Commons Attribution (CC-BY) licence to any Author Accepted Manuscript version arising.

\bibliographystyle{plainurl}
{\small
\bibliography{ref}
}

\appendix

\section{Proof of \texorpdfstring{\Cref{thm:luczaknew}}{core theorem}}
\label{appendix}

To prove \Cref{thm:luczaknew}, we translate a number of results from the random graph model $G(n,p)$ to Erd\H{o}s--R\'enyi random graph processes.

A graph property $Q$ is said to be \emph{monotonic increasing} if $Q$ holds for $G$ implies that $Q$ holds for $G+e$ for any choice of non-edge $e$ of $G$.
Similarly, $Q$ is said to be \emph{monotonic decreasing} if $Q$ holds for $G$ implies that $Q$ holds for $G-e$ for any choice of edge $e$ of $G$.
We will call a graph property \emph{monotonic} if it is monotonic increasing or decreasing.
Monotonic properties that are satisfied by $G(n,p)$ can be translated to $G(n,M)$ as follows.

\begin{corollary}\label{cor:monotone+simple}
    Let $Q$ be a monotonic graph property and suppose $m(n) = \left\lfloor \binom{n}{2} p(n) \right\rfloor \rightarrow \infty$ as $n \rightarrow \infty$.
    If $Q$ is monotonic, then
        \begin{equation*}
        \lim_{n \rightarrow \infty} \mathbb{P}[G(n,p(n)) \text{ satisfies } Q] = 1 \implies \lim_{n \rightarrow \infty} \mathbb{P}[G(n,m(n)) \text{ satisfies } Q] = 1.
        \end{equation*}
\end{corollary}
\begin{proof}
    This is an immediate consequence of parts (i) and (ii) of \cite[Corollary 1.16]{JansonLuczakRucinski}.
\end{proof}

It follows from results of {\L}uczak \cite{luczak} that for any fixed $k \geq 3$, the random graph $G(n,p)$ has a $k$-connected $k$-core containing at least $5n/9$ vertices when $p(n) \geq (\log n)/n$ a.a.s.
It now seems immediate how to prove \Cref{thm:luczaknew}:
we apply \Cref{cor:monotone+simple} to the property ``$G$ contains a $k$-connected $k$-core containing at least $5n/9$ vertices'', with the knowledge that $M_1(\omega) \geq m(n)$ a.a.s.~by \cite[Theorem 1]{erdosrenyi} and \cite[Theorem 4.2]{Frieze_K_2015}.
Unfortunately, this property is not monotonic, as adding edges can reduce the connectivity of the $k$-core.
Fortunately, we can combine four different monotonic properties (3 increasing and 1 decreasing) to prove \Cref{thm:luczaknew}.


			
            


The first of our monotone graph properties, which is increasing, is
that of having a large $k$-core:\vspace{2mm}

\noindent\textbf{Property $P_{k,\varepsilon}$:}
The $k$-core of $G=([n],E)$ contains at least $n-n\exp(-\log^{\varepsilon}(n))$ vertices.

\begin{lemma}\label{lem:luczak0}
    For any fixed integer $k \geq 3$ and fixed $0<\varepsilon <1/2$, $G(n,M_2(\omega),\omega)$ satisfies Property $P_{k,\varepsilon}$ a.a.s.
\end{lemma}

\begin{proof}
    \cite[Theorem 2]{luczak} states that
    for every $\varepsilon >0$ there exists some constant $d>0$ such that for $c(n) = np > d$ and $k \leq c - c^{0.5 + \varepsilon}$, the $k$-core of $G(n,p)$ has at least $n - n \exp(-c^\varepsilon)$ vertices a.a.s.
    Setting $p := (\log n)/n$ and applying this theorem, we see that $G(n,p)$ satisfies Property $P_{k,\varepsilon}$ a.a.s.

    To get a result for $G(n,M_2(\omega),\omega)$, observe that $P_{k,\varepsilon}$ is monotonically increasing and 
    \[
    m(n) := \left\lfloor \binom{n}{2} p(n) \right\rfloor = \frac{1}{2}(n-1)\log n \longrightarrow \infty \qquad \text{as } n \longrightarrow \infty \, .
    \]
    Applying \Cref{cor:monotone+simple}, we have that $\mathbb{P}[G(n,m(n)) \text{ satisfies } P_{k,\varepsilon}] \rightarrow 1$ as $n \rightarrow \infty$.
    By \cite[Theorem 4.3]{Frieze_K_2015},
    $G(n,m(n))$ is not 2-connected a.a.s.,
    which implies $m(n) < M_2(\omega)$ a.a.s.
    Property $P_{k,\varepsilon}$ being increasing now implies $G(n,M_2(\omega),\omega)$ satisfies Property $P_{k,\varepsilon}$ a.a.s.
\end{proof}

We next consider the increasing graph property that every pair of large vertex sets must have an edge between them.
\vspace{2mm}

\noindent\textbf{Property $Q_{k,\varepsilon}$:}
For every pair of disjoint vertex sets $S_0, S_1$ of $G=([n],E)$ satisfying
\begin{equation*}
     |S_0| \leq k-1 + n \exp(-\log^{\varepsilon}(n)), \qquad n \log^{-6}(n) \leq |S_1| \leq  n/2,
\end{equation*}
there exists an edge $v_1v_2 \in E$ with $v_1 \in S_1$ and $v_2 \in (S_0 \cup S_1)^c$.

\begin{lemma}\label{lem:luczak4}
    For any fixed integer $k \geq 3$ and fixed $0 < \varepsilon < 1$, $G(n,M_2(\omega),\omega)$ satisfies Property $Q_{k,\varepsilon}$ a.a.s.
\end{lemma}

\begin{proof}
    Let $p=\frac{\log(n)}{n}$,
    we show that $G(n,p)$ satisfies $Q_{k,\varepsilon}$ a.a.s.
    The proof that $G(n,M_2(\omega),\omega)$ satisfies $Q_{k,\varepsilon}$ then follows from the same application of \Cref{cor:monotone+simple} as in the proof of \Cref{lem:luczak0}.

    Let $A_{r,s}$ be the number of disjoint vertex set pairs $(S_0,S_1)$ with $|S_1|= r$ and $|S_0|=s$ such that there is no edge connecting $S_1$ and $(S_0 \cup S_1)^c$.
    Fix $a = n \log^{-6}(n)$ and $b = k-1 +  n\exp(-\log^{\varepsilon}(n))$.
    Write $e(X,Y)$ for the event that there are no edges between $X$ and $Y$, and recall that for large $n$:
    \begin{equation}\label{eq:facts}
      \binom{n}{m} < \biggl(\frac{e \cdot n}{m}\biggr)^{\!m},\qquad \left(\exp(1+\log^\varepsilon n) \right)^{k-1} < (e \cdot n)^{k-1},
      \qquad
      \left( 1-\frac{\log n}{n} \right)^n < e^{-\log(n)} = \frac{1}{n}.
    \end{equation}
    Using the first moment method we see that, for large $n$,
    {\allowdisplaybreaks
    \begin{align*}
        \mathbb{P} \left[ \sum_{r=\lceil a \rceil}^{\lfloor n/2 \rfloor} \sum_{s=0}^{\lfloor b \rfloor} A_{r,s} > 0  \right]
        &\leq \mathbb{E} \left[ \sum_{r=\lceil a \rceil}^{\lfloor n/2 \rfloor} \sum_{s=0}^{\lfloor b \rfloor} A_{r,s}  \right]
        = \sum_{r=\lceil a \rceil}^{\lfloor n/2 \rfloor} \sum_{s=0}^{\lfloor b \rfloor} \mathbb{E} \left[  A_{r,s}  \right] \\
        &= \sum_{r=\lceil a \rceil}^{\lfloor n/2 \rfloor} \sum_{s=0}^{\lfloor b \rfloor} \binom{n}{r}\binom{n-r}{s} \mathbb{P} \Big[ e(S_1 \, , \,(S_0 \cup S_1)^c)  ~ | ~ |S_1| = r , ~ |S_0| = s, ~ S_0 \cap S_1 = \emptyset \Big]\\
        &= \sum_{r=\lceil a \rceil}^{\lfloor n/2 \rfloor} \sum_{s=0}^{\lfloor b \rfloor} \binom{n}{r}\binom{n-r}{s} \left( 1-\frac{\log n}{n} \right)^{r(n-r-s)} \\
        &< O(n) \sum_{r=\lceil a \rceil}^{\lfloor n/2 \rfloor} \binom{n}{r} \binom{n}{b} \left( 1-\frac{\log n}{n} \right)^{r(n-r-b)} \\
        &< O(n) \sum_{r=\lceil a \rceil}^{\lfloor n/2 \rfloor} \left(\frac{n \cdot e}{b} \right)^b \left(\frac{n \cdot e}{r} \left( 1-\frac{\log n}{n} \right)^{n-r-b}\right)^r  \\
        &< O(n) \sum_{r=\lceil a \rceil}^{\lfloor n/2 \rfloor} \left(\frac{n \cdot e}{n\exp(-\log^\varepsilon n)} \right)^{k-1 + n \exp(-\log^\varepsilon n)} \left(\frac{n \cdot e}{r} \left( 1-\frac{\log n}{n} \right)^{n-r-b}\right)^r \\
        &= O(n) \sum_{r=\lceil a \rceil}^{\lfloor n/2 \rfloor} \left(\exp(1+\log^\varepsilon n) \right)^{k-1 + n \exp(-\log^\varepsilon n)} \left(\frac{n \cdot e}{r} \left( 1-\frac{\log n}{n} \right)^{n-r-b}\right)^r \\
        &< O(n) \sum_{r=\lceil a \rceil}^{\lfloor n/2 \rfloor} (e \cdot n)^{k-1}\left(\exp(1+\log^\varepsilon n) \right)^{n \exp(-\log^\varepsilon n)} \left(\frac{n \cdot e}{a} \left( 1-\frac{\log n}{n} \right)^{n/3}\right)^r \\
        &= O(n^k) \sum_{r=\lceil a \rceil}^{\lfloor n/2 \rfloor} \left(\exp(1+\log^\varepsilon n) \right)^{n \exp(-\log^\varepsilon n)} \left(e^3 \cdot \log^{18}(n) \cdot \left( 1-\frac{\log n}{n} \right)^{n}\right)^{r/3}\\
        &< O(n^k) \sum_{r=\lceil a \rceil}^{\lfloor n/2 \rfloor} \left(\exp(1+\log^\varepsilon n) \right)^{n \exp(-\log^\varepsilon n)} \left(e^{3} \cdot \frac{\log^{18}(n)}{n} \right)^{r/3} \\
        &< O(n^{k+1}) \left(\exp(1+\log^\varepsilon n) \right)^{n \exp(-\log^\varepsilon n)} \left(e^{3} \cdot \frac{\log^{18}(n)}{n} \right)^{\frac{1}{3}n\log^{-6}n}
    \end{align*}}
    where the fifth, eighth and tenth lines use the first, second and thirds parts of \cref{eq:facts} respectively.
    Applying a logarithm to the last term gives
    \begin{align*}
        O(1) + (k+1)\log n + n\left( \frac{1+\log^\varepsilon n}{\exp(\log^\varepsilon n)} + \frac{3 + 18\log \log n - \log n}{3\log^{6}n} \right) \; \sim \; - \frac{1}{3} n \log^{-5} n  \; \longrightarrow  \; -\infty.
    \end{align*}
    This now implies our desired probability converges to 0 as $n \rightarrow \infty$.
\end{proof}

Finally, we consider the monotonic graph property that small induced subgraphs are sparse.\vspace{2mm}

\noindent\textbf{The Sparsity Property:}
For every vertex subset $S \subseteq [n]$ of $G$ with induced edge set $E(S)$,
\begin{enumerate}
    \item if $|S| < 500$ then $|E(S)| \leq |S|$;
    \item if $|S| \leq 2 n \log^{-6}(n)$, then $|E(S)| < (1.25)|S|$.
\end{enumerate}
Unlike Properties $P_{k,\varepsilon}$ and $Q_{k,\varepsilon}$,
the Sparsity Property is decreasing.

\begin{lemma}\label{lem:luczak5}
    For any fixed integer $k \geq 3$, $G(n,M_k(\omega),\omega)$ satisfies the Sparsity Property a.a.s.
\end{lemma}

\begin{proof}
    Let $p = (\log n + k \log\log n)/(n-1)$.
    We now show $G(n,p)$ satisfies the Sparsity Property a.a.s.
    For (i), let $H$ be a fixed subgraph with $d(H) := |E(H)|/|V(H)| > 1$.
    As $p = o(n^{-1/d(H)})$, then by \cite[Theorem 5.2]{Frieze_K_2015} $G(n,p)$ does not contain $H$ a.a.s.
    If we apply this result for all graphs on less than 500 vertices with more edges than vertices, we get (i).
    For (ii), applying the unique lemma of \cite{luczak} with $a = 1.25$ and $c = n(\log(n) + k \log\log n)/(n-1)$ gives that a.a.s.~every vertex subset $S$ satisfying
    \[
        0.35\left(\frac{2.5}{c}\right)^5 e^{-8} n > 0.01 c^{-5} n > 2n \log^{-6}(n) \geq |S| \, ,
    \]
    also satisfies $|E(S)| < 1.25|S|$.
    Hence the Sparsity Property holds for $G(n,p)$ a.a.s.

    As the Sparsity Property is decreasing, we apply \Cref{cor:monotone+simple} to obtain a statement for $G(n,M,\omega)$.
    Explicitly, we have $G(n,m(n))$ a.a.s.~satisfies the Sparsity Property where
    \[
    m(n) = \binom{n}{2}p = \frac{n}{2}(\log n + k \log\log n) \longrightarrow \infty \qquad \text{as } n \longrightarrow \infty \, .
    \]
    By \cite[Theorem 4.3]{Frieze_K_2015},
    $G(n,m(n))$ is $k$-connected a.a.s.,
    which implies $m(n) \geq M_k(\omega)$ a.a.s.
    Now, since the Sparsity Property is decreasing, $G(n,M_k(\omega),\omega)$ satisfies the Sparsity Property a.a.s.
\end{proof}

Finally, we also require the following classical result of Bollobás and Thomason for our final monotonic property: $k$-connectivity.

\begin{theorem}[{\cite[Theorem 4]{BolloThomason}}]\label{thm:bollobasthompson}
    For any fixed integer $k \geq 1$, $G(n,M_k(\omega),\omega)$ is $k$-connected a.a.s.
\end{theorem}

\begin{proof}[Proof of \Cref{thm:luczaknew}]
    Fix $\varepsilon$ to any positive constant; e.g.~$\varepsilon = 0.1$.
    By \Cref{lem:luczak0}, \Cref{lem:luczak4}, \Cref{lem:luczak5} and \Cref{thm:bollobasthompson}, and using the monotonicity of each property,
    a.a.s.~our random choice of ordering $\omega$ of $E(K_n)$ satisfies the following properties for each graph $G(n,M,\omega)$:
    \begin{enumerate}
        \item if $M \geq M_2$ then $G(n,M,\omega)$ satisfies Properties $P_{k,\varepsilon}$ and $Q_{k,\varepsilon}$;
        \item if $M \leq M_k$ then $G(n,M,\omega)$ satisfies the Sparsity Property;
        \item if $M \geq M_k$ then $G(n,M,\omega)$ is $k$-connected (and hence is equal to its $k$-core).
    \end{enumerate}
    It suffices to show any graph $G=(V,E)$ with sufficiently large $n:= |V|$ that satisfies Properties $P_{k,\varepsilon}$ and $Q_{k,\varepsilon}$ and the Sparsity Property has a $k$-connected $k$-core containing at least $5n/9$ vertices.

    Since $\exp(-\log^{\varepsilon}(n)) \rightarrow 0$ as $n \rightarrow \infty$ and $n$ has been chosen to be sufficiently large,
    we have $n-n\exp(-\log^{\varepsilon}(n)) \geq 5n/9$.
    Property $P_{k,\varepsilon}$ now guarantees that the $k$-core of $G$ contains at least $5n/9$ vertices.
    It now remains to prove that the $k$-core of $G$ is $k$-connected.

    Suppose for contradiction that the $k$-core of $G$ is not $k$-connected.
    Then there exists a partition $S_1,S_2,T,U$ of the vertices of $G$ so that $S_1 \cup S_2 \cup T$ induces the $k$-core of $G$, $|T| \leq k-1$ and there are no edges connecting the sets $S_1,S_2$.
    We pick $T$ to be minimal with respect to this property: in particular, we may assume that every vertex in $T$ is adjacent to vertices in both $S_1$ and $S_2$.
    By relabelling we may suppose that $|S_2| \geq |S_1|$.
    As $S_1, S_2, T$ is a vertex partition of the $k$-core,
    we have that $|S_1|+|S_2|+|T| \geq n- n \exp (-\log^\varepsilon(n))$.

    \begin{claim}\label{claim:sparsity}
        For all $k \geq 3$, the induced subgraph $G[S_1 \cup T]$ has at least $1.25(|S_1| + |T|)$ edges.
    \end{claim}
    \begin{claimproof}
    By our assumption that the $k$-core is not $k$-connected, we must have $S_1 \neq \emptyset$.
    Moreover, as every vertex in $S_1$ has degree at least $k$ in $G[S_1 \cup T]$ and $|T| \leq k-1$, we have $|S_1| \geq 2$.
    Let $v,w$ be distinct vertices in $S_1$.
    By the Sparsity Property, $G$ contains no subgraph on less than 500 vertices that has more edges than one plus the number of vertices.
    Thus $v,w$ must have at most two common neighbours in $G[S_1 \cup T]$.
    Moreover, if they have two common neighbours then $vw\notin E(G)$.
    As every vertex in $S_1$ has degree at least $k$ in $G[S_1 \cup T]$, we have
    \[
    |N_G(v) \cap N_G(w)| \geq 
    \begin{cases}
    2(k-1) -1 & vw \in E(G) \\
    2k -2 & vw \notin E(G)
    \end{cases} \, .
    \]
    When combined with $|S_1| \geq 2$, it follows that the number of vertices of $V[S_1 \cup T]$ is at least
    \begin{equation*}\label{eq:S-geq-T}
        |S_1 \cup T| \geq 2k - 1 \qquad \implies \qquad |S_1| \geq 2k - 1 - |T| \geq k > |T| \, .
    \end{equation*}
    We will also repeatedly use the fact that the number of edges in $G[S_1 \cup T]$ is at least
    \[
    \big|E[S_1 \cup T]\big| = \frac{1}{2}\sum_{v \in S_1 \cup T} \deg_{G[S_1 \cup T]}(v) \geq \frac{1}{2}\sum_{v \in S_1} \deg_{G[S_1 \cup T]}(v) \geq \frac{1}{2}k|S_1| \, .
    \]
    Observe that 
    \begin{align} \label{eq:1.25+edges}
        \frac{1}{2}k|S_1| \geq \frac{5}{4}(|S_1| + |T|) \quad \Longleftrightarrow \quad |S_1| \geq \frac{5}{2k-5}|T| \, .
    \end{align}
    As $|S_1| > |T|$, \eqref{eq:1.25+edges} always holds for $k\geq 5$, and hence $G[S_1 \cup T]$ has at least $1.25(|S_1| + |T|)$ edges.
    For the remaining cases of $k =3,4$, we observe that 
    \begin{align} \label{eq:small+k}
    \frac{1}{2}k|S_1| > |S_1| + (k-1) \quad \Longleftrightarrow \quad |S_1| > 2 + \frac{2}{k-2} \, .    
    \end{align}
    If the right-hand side of \eqref{eq:small+k} holds then $G[S_1 \cup T]$ has more edges than vertices.
    In this case the Sparsity Property gives that $|S_1 \cup T| \geq 500$, which in turn implies $|S_1|\geq 495$ as $|T| \leq k-1 \leq 3$.
    These parameters also satisfy \eqref{eq:1.25+edges} when $k=3,4$, and so $G[S_1 \cup T]$ has at least $1.25(|S_1| + |T|)$ edges if the right-hand side of \eqref{eq:small+k} holds.
    We note that as $|S_1| \geq k$, the right-hand side of \eqref{eq:small+k} always holds when $k=4$.
    
    The only remaining cases are where $k=3$ and $|S_1| \in \{3,4\}$.
    By the Sparsity Property, in these cases $G[S_1 \cup T]$ can have at most $|S_1 \cup T|$ edges.
    Moreover, our assumption on the minimality of $T$ implies each $t \in T$ has degree at least one.
    Hence, by a degree counting argument, we have 
    \[
    |S_1| + |T| \geq \big|E[S_1 \cup T]\big| =\frac{1}{2} \sum_{v \in S_1 \cup T} \deg_{G[S_1 \cup T]}(v) \geq  \frac{3}{2}|S_1| + \frac{1}{2} |T| \, .
    \]
    This implies $|S_1| \leq |T|$, giving a contradiction.
    \end{claimproof}

    \Cref{claim:sparsity} along with the Sparsity Property and $|S_1| > |T|$ now imply that
    \begin{equation*}
        |S_1| + |T| > 2 n \log^{-6}(n) \qquad \implies \qquad |S_1| > \frac{1}{2}\Big(|S_1| + |T|\Big) > n \log^{-6}(n).
    \end{equation*}
    Now fix $S_0 = T \cup U$. Then $|S_0| \leq k-1 + n\exp(-\log^\varepsilon(n))$.
    However, there are no edges between $S_1$ and $S_2$, contradicting Property $Q_{k,\varepsilon}$.
    It follows that the $k$-core of $G$ is $k$-connected, which concludes the proof.
\end{proof}

\end{document}